\documentclass[11pt,twoside]{article}
\usepackage{amsmath,amsbsy,amsfonts}
\usepackage{graphicx,color}
\everymath{\displaystyle}
\newtheorem{theorem}{Theorem}[section]
\newtheorem{lemma}[theorem]{Lemma}
\newtheorem{proposition}[theorem]{Proposition}
\newtheorem{corollary}[theorem]{Corollary}

\newtheorem{claim}[theorem]{Claim}

\newenvironment{proof}[1][Proof]{\noindent\textbf{#1.} }{\ \rule{0.5em}{0.5em}}

\begin{document}

\title{\bf Nodal ground state solution to a biharmonic equation via dual method}

\author{Claudianor O. Alves\thanks{C. O. Alves was partially supported by CNPq/Brazil
		301807/2013-2 and INCT-MAT, coalves@dme.ufcg.edu.br}\, , \, \  Al\^{a}nnio B. N\'{o}brega \thanks{alannio@dme.ufcg.edu.br}\vspace{2mm}
	\and {\small  Universidade Federal de Campina Grande} \\ {\small Unidade Acad\^emica de Matem\'{a}tica} \\ {\small CEP: 58429-900, Campina Grande - Pb, Brazil}\\}

\date{}
\maketitle

\begin{abstract}
	Using dual method we establish the existence of nodal ground state solution for the
	following class of problems 
	$$
	\left\{
	\begin{array}{l}
	\Delta^2 u   =  f(u), \quad \mbox{in} \quad \Omega, \\
	u =Bu=0,\quad\mbox{on} \quad \partial \Omega
	\end{array}
	\right.
	$$
	where $\Delta^2$ is the biharmonic operator, $B=\Delta$ or $B=\dfrac{\partial}{\partial \nu}$ and $f$ is a $C^1-$
	function having subcritical growth.

	\vspace{0.3cm}
	
	
	\vspace{0.3cm}
	
\end{abstract}

\section{Introduction}

In this paper,  we are concerned with the existence of {\it nodal ground state solution} for the following problem
$$
	\left\{
	\begin{array}{l}
	\Delta^2 u   =  f(u), \quad \mbox{in} \quad \Omega, \\
	u =Bu=0,\quad\mbox{on} \quad \partial \Omega
	\end{array}
	\right.
	\eqno{(P)}
	$$
	where $\Delta^2$ is the biharmonic operator, $f$ is a $C^{1}-$function with subcritical growth and  $Bu=\Delta u$  or $Bu=\dfrac{\partial u}{\partial \nu}$. If $Bu=\Delta u$, we have the {\it Navier boundary conditions}
$$
	u=\Delta u=0,\ \mbox{on}\ \partial\Omega,
$$
and for the case $Bu=\dfrac{\partial u}{\partial \nu}$, we have the boundary condition 
$$
	u= \frac{\partial u}{\partial \nu}=0,\ \mbox{on}\ \partial\Omega,
$$
which is called {\it Dirichlet boundary conditions}. Here $\frac{\partial }{\partial \nu}$ denotes the exterior normal derivative at the boundary. Hereafter, in the case of Dirichlet  boundary condition, we assume that $\Omega$ permits to apply maximum principle, for more details about this subject, see Gazzola, Grunau and Sweers \cite[Chapter 6]{G}, and Grunau and Robert \cite{GR}.

In what follows, we say that a solution $u$ of $(P)$ is a {\it nodal solution}, when $u^{\pm}\not=0$, where $u^{+}=\max\{u,0\}$ and $u^{-}=\min\{u,0\}$. 

Related to nonlinearity $f$, we assume the following assumptions:
\begin{description}
	\item[$(f_1)$] $f:\mathbb{R} \rightarrow \mathbb{R}$ is a $C^1$ function  and $f(0)=f'(0)=0.$
	\item[$(f_2)$] $f$ is odd, that is, $f(t)=-f(-t), \quad \forall t \in \mathbb{R}.$
	\item[$(f_3)$] There exist $c_0>0$ and $p \in (2,2_*),$ such that 
	$$
	\lim_{t\rightarrow +\infty}\frac{f(t)}{t^{p-1}}=c_0,
	$$ 
	where
	$$
	2_*=
	\left\{ \begin{array}{c}
	\frac{2N}{N-4},\quad N\geq 5 \\
	+\infty, \quad 1\leq N \leq 4 .
	\end{array}
	\right.
	$$ 
	\item[$(f_4)$] There exist $b_0>0$ and $q \in (2,p]$, such that 
	$$
	\lim_{t\rightarrow 0^+}\frac{f(t)}{t^{q-1}}=b_0.
	$$
	\item[$(f_5)$] $\frac{f(t)}{t}$ is increasing for $t>0$.
\end{description}

Here, we would like point out that the function below verifies the conditions $(f_1)-(f_5)$:
$$
f(t)=\sum_{j=1}^{k}a_j|t|^{p_j-2}t, \quad \forall t \in \mathbb{R},
$$
where $a_j>0$ and $p_j \in (2,2_{*})$ for all $j \in \{1,....,k\}$.

 The equations involving the biharmonic operator have received special attention of many researchers, in part, because describe the mechanical vibrations of an elastic plate, which among other things describes the traveling waves in a suspension bridge, see \cite{BDF,FG,G,GK,LM}. Moreover, the biharmonic operator has  intrinsic problems as, the lack of a maximum principle for all bounded domains. Recently, many authors have studied various aspects of the biharmonic, see for example, \cite{BG,JQ,P3,P1,P2,YT,ZWZ}. 
 
 In the case of the Laplacian operator, the study of the existence of nodal solution has a rich literature, see for example,  Bartsch, Weth and Willem \cite{Weth}, Bartsch and Weth \cite{Bartsch}, Bartsch, Liu and Weth \cite{Bartsch1}, Castro, Cossio and Neuberger \cite{Castro} and their references. However, we cannot use or adapt some techniques developed for the laplacian, because in the most part of the above papers, the authors prove the existence of nodal solution for problems like  
 $$
 \left\{
 \begin{array}{rcl}
 -\Delta u = f(u) \ \ \mbox{in} \ \ \Omega,\\
 u=0 \ \ \mbox{on} \ \partial\Omega,
 \end{array}
 \right.\leqno{(E)}
 $$
 by minimizing the energy function
 $J:H^{1}_{0}(\Omega) \to \mathbb{R}$ given by
 $$
 J(u)= \frac{1}{2}\int_{\Omega}|\nabla
 u|^{2} \, dx -\int_{\Omega} F(u) \ dx,
 $$
 on the set 
 $$
 \mathcal{M}=\{u \in H^{1}_{0}(\Omega)\,:\, J'(u^{\pm})u^{\pm}=0\}.
 $$
 After some estimates, it is proved that there is $u \in \mathcal{M}$ such that $J'(u)=0$. This critical point is called a 
 {\it nodal ground state solution ( or least energy nodal solution ) } for $(E)$.  In problems involving the biharmonic operator, we cannot even ensure that given  $u \in H^{2}(\Omega)$, we also have  $u^{\pm} \in H^{2}(\Omega)$.
 
The existence of nodal solution  for $(P)$ has been studied by Weth \cite{Weth2}, by supposing the following conditions on $f$: \\

\noindent $(W_1)$ \,  $f:\Omega \times \mathbb{R} \to \mathbb{R}$ is a Carathe\'odory function, and $f(x,0)=0$ for a.e $x$ in $\Omega$.\\
\noindent $(W_2)$ \, There are $q^{*}>0, q_* \in (0,\lambda_1)$, and $0<p<\frac{8}{N-4}$ for $N>4$, resp. $p>0$ for $N \leq 4$, such that
$$
|f(x,t)-f(x,s)|\leq [q_*+q_*(|t|^{p}+|s|^{p})]|t-s| \quad \mbox{for a.e} \quad x \in \Omega, \quad t \in \mathbb{R}.
$$
\noindent $(W_3)$ \, There are $R>0$ and $\eta > 2$ such that
$$
\eta F(x,t) \leq f(t)t, \quad \mbox{for a.e} \quad x \in \Omega, \, |t| \geq R.
$$
\noindent $(W_4)$ \, $f$ is nondecreasing in $t \in \mathbb{R}$ for a.e $x \in \Omega$.

Here, $F(x,t)=\int_{0}^{t}f(x,s)ds$ and $\lambda_1$ denotes the first eigenvalue of $\Delta^{2}$ on $\Omega$ relative to the Dirichlet or Navier boundary conditions. In that paper,  using the Moreau's decomposition for a Hilbert space,  Weth has showed  the existence of at least three  solutions, a positive solution, a negative solution and a nodal solution.

Motivated by the above references, in the present paper we study the existence of nodal solution for problem $(P)$ using a new approach, more precisely, the {\it Dual Method}. Here, we have completed the study made in \cite{Weth2}, in the following sense: \\
\noindent {\bf 1}- Our arguments permit to consider some nonlinearities, which cannot be used in \cite{Weth2}. For example, be we can work with a nonlinearity like
$$
f(t)=\varphi(t)|t|^{p-2}t,
$$
where $\varphi$ is a $C^1-$function, increasing, positive and bounded such that for any $s>1$, the function $\frac{f'(t)}{t^{s-2}}$ is not bounded at infinity. However, this type of nonlinearity cannot be used in \cite{Weth2}, because $(W_2)$ yields  $\frac{f'(t)}{t^{p-2}}$ is bounded at infinity. \\
\noindent {\bf 2}- Our main result establishes the existence of nodal ground state solution, which was not considered in  \cite{Weth2}

 \vspace{0.5 cm}

Before to state our main result, we would like to recall that the energy functional $I:H \rightarrow \mathbb{R}$ associated with $(P)$ is given by 
$$
I(u)=\frac{1}{2}\int_{\Omega}|\Delta u|^2dx-\int_{\Omega}F(u) dx,
$$
where $H=H^2(\Omega)\cap H_0^1(\Omega)$ in the case of the Navier boundary condition, and $H=H_0^2(\Omega)$ for the Dirichlet boundary condition. Moreover, it is well known that for these boundary conditions, $H$ is a Hilbert space endowed with the inner product 
$$
\left\langle u,v\right\rangle=\int_{\Omega}\Delta u \Delta v \, dx,
$$
whose associated norm is given by
$$
\|u\|=\left(\int_{\Omega}|\Delta u|^2dx\right)^{\frac{1}{2}}.
$$

It is standard to check that critical points of $I$ are precisely weak solutions of $(P)$. In the sequel, we will say that $u \in H$ is a {\it nodal ground state solution} if 
$$
I(u)=\min\{I(v)\,:\, v \quad \mbox{is a nodal solution for} \quad (P) \}.
$$

Our main result is the following
 \begin{theorem}\label{Main Theorem}
  Suppose that $f$ satisfies $(f_1) -(f_5)$. Then, problem $(P)$ possesses a nodal ground state solution.
 \end{theorem}

Before to conclude this introduction, we would like point out that the Dual Method have been used to study the existence of solution for a lot of  types of problems, for example,  elliptic equations, elliptic systems, and  wave equations. The reader can get more information about this method in the papers due to Alves, Carri\~ao and Miyagaki  \cite{ACM}, Alves \cite{Alves2}, Ambrosetti and Struwe \cite{AS}, Struwe \cite{Struwe}, Willem \cite{W1} and their references.

\section{The Dual Method}

In this section, we will define and show some properties of the dual functional associated with $(P)$. To this end, we begin recalling that for each $w \in L^{\frac{p}{p-1}}(\Omega)$, using some results found in Gazzola, Grunau and Sweers \cite[Chapter 2]{G}, there is a unique solution $u \in W^{4,\frac{p}{p-1}}(\Omega)$ of the linear problem  
$$
 \left\{ 
\begin{array}{cc}
 \Delta^2 u= w,&\ \mbox{in}\ \Omega; \\
 u=B u=0 &\ \mbox{on}\ \partial \Omega.
 \end{array}
 \right.
\eqno{(P_w)}
$$
Moreover, there is a constant $C>0$ such that
$$
\|u\|_{W^{4,\frac{p}{p-1}}(\Omega)} \leq C \|w\|_{L^{\frac{p}{p-1}}(\Omega)}.
$$

From the above commentaries, we can consider the linear operator \linebreak $T:L^{\frac{p}{p-1}}(\Omega) \rightarrow W^{4,\frac{p}{p-1}}(\Omega)$, such that for $w \in L^{\frac{p}{p-1}}(\Omega)$,  $Tw$ is the unique solution of $(P_w)$. From the last inequality,  
$$
\|Tw\|_{W^{4,\frac{p}{p-1}}(\Omega)} \leq C \|w\|_{L^{\frac{p}{p-1}}(\Omega)}, \quad \forall w \in L^{\frac{p}{p-1}}(\Omega),
$$
showing that $T$ is continuous. Now, recalling that the embeddings below
$$
W^{4,\frac{p}{p-1}}(\Omega) \hookrightarrow L^{s}(\Omega), \quad \forall s \in \left[\frac{p}{p-1},4(p)_*\right)
$$
are compact for 
$$
	4(p)_*=
	\left\{ \begin{array}{c}
	\frac{Np}{p(N-4)-N},\quad N\geq 5 \\
	+\infty, \quad 1\leq N \leq 4 ,
	\end{array}
	\right.
$$ 
we can ensure that $T:L^{\frac{p}{p-1}}(\Omega) \rightarrow L^{p}(\Omega)$ is a linear compact operator, because $p \in \left(\frac{p}{p-1},4(p)_*\right)$. Moreover, $T$ satisfies the following properties: \\

\noindent $(T_1)$ \, $T$ is {\it positive}, that is, for any $w \in L^{\frac{p}{p-1}}(\Omega)$, $\int_{\Omega}wTw\,dx\geq 0$. Moreover, if $w$ is nonnegative and $w \not=0$, $Tw>0$ in $\Omega$. \\

\noindent $(T_2)$ \, $T$ is {\it symmetric}, in the sense that, if $w_1,w_2 \in L^{\frac{p}{p-1}}(\Omega)$, then  
$$
\int_{\Omega}w_1Tw_2\,dx= \int_{\Omega}w_2Tw_1\,dx.
$$ 

Using the functional $T$, we set $\Psi:L^{\frac{p}{p-1}}(\Omega) \to \mathbb{R}$ by
$$
\Psi(w)=\int_{\Omega}H(w)dx-\frac{1}{2}\int_{\Omega}wTw dx,
$$
where $H(t)=\int_{0}^{t}h(s)ds$ and $h$ is the inverse of $f$. Note that, $f$ is invertible, because $(f_1)-(f_5)$ imply that $f:\mathbb{R} \to \mathbb{R}$ is bijective. The functional $\Psi$ is called the {\it dual functional} associated with $(P)$.

In the sequel, for any $w \in L^{\frac{p}{p-1}}(\Omega)$, we will denote by $\|w\|$ its norm in $L^{\frac{p}{p-1}}(\Omega)$, that is,
$$
\|w\|=\left(\int_{\Omega}|w|^{\frac{p}{p-1}}\,dx \right)^{\frac{p-1}{p}}.
$$

Next, we will prove some properties of $\Psi$. However, to do this, firstly we must show some properties of $h$. \\

\noindent ${\bf (h_0)}$ \, $h$ is continuous, $h(0)=0$ and $h(t)=-h(-t) \quad \forall t \in \mathbb{R}$ . \\ 

\noindent ${\bf (h_1)}$ $h$ verifies the following growth conditions: Given $\varepsilon>0 $, there are $ \delta, M >0$ such that
\begin{equation}\label{1}
\left(\frac{1}{c_0(1+\varepsilon)}t\right)^{1/(p-1)}\le h(t)\le \left(\frac{1}{c_0(1-\varepsilon)}t\right)^{1/(p-1)},\, \forall t \ge f(M),
\end{equation}
and
\begin{equation}\label{2}
\left(\frac{1}{b_0(1+\varepsilon)}t\right)^{1/(q-1)}\le h(t)\le \left(\frac{1}{b_0(1-\varepsilon)}t\right)^{1/(q-1)},\, \forall t \le f(\delta) .
\end{equation}
\noindent Indeed, from $(f_3)-(f_4)$,  given $\varepsilon>0 $, there are $ \delta, M >0$ such that 
\begin{equation}\label{3}
(1-\varepsilon)c_0t^{p-1}\le f(t)\leq (1+\varepsilon)c_0t^{p-1},\, \forall t \ge M 
\end{equation}
and 
\begin{equation}\label{4}
(1-\varepsilon)b_0t^{q-1}\le f(t)\leq (1+\varepsilon)b_0t^{q-1},\, \forall t \le \delta.
\end{equation}
Now, (\ref{1})-(\ref{2}) follow from $(\ref{3})-(\ref{4})$.  \\

\noindent ${\bf (h_2)}$ The functions $H$ and $h$ satisfy the following inequality
\begin{equation}\label{5}
H(t)-\frac{1}{2}h(t)t\geq C_{\varepsilon} t^{p/(p-1)}, \quad \forall t \geq f(M).
\end{equation}
\noindent In fact, for $t \geq f(M)$,   
$$ 
H(t)\geq \left( \frac{p-1}{p}\right)\left( \frac{1}{c_0 (1+\varepsilon)}\right)^{1/(p-1)} t^{p/(p-1)}+K
$$
where  $K$ is a constant, which can be negative. Once 
$$
\limsup_{t \rightarrow +\infty}\frac{H(t)}{t^{p/(p-1)}}\geq \left( \frac{p-1}{p}\right)\left( \frac{1}{c_0(1+\varepsilon)}\right)^{1/(p-1)},
$$
we derive 
$$
H(t)\geq \left[ \left( \frac{p-1}{p}\right)\left( \frac{1}{c_0(1+\varepsilon)}\right)^{1/(p-1)}-\varepsilon\right] t^{p/(p-1)}.
$$ 
for $t$ large enough. Hence, by (\ref{1}), 
\begin{equation}\label{6}
	H(t)-\frac{1}{2}h(t)t \geq C_\varepsilon t^{p/(p-1)},
\end{equation}
for $t$ large enough and 
$$
C_\varepsilon = \left[ \frac{(p-1)}{p}\left( \frac{1}{c_0(1+\varepsilon)}\right)^{1/(p-1)}- \frac{1}{2}\left( \frac{1}{c_0(1-\varepsilon)}\right)^{1/(p-1)}-\varepsilon\right].
$$ 
As $C_\varepsilon>0$ for $\varepsilon$ small enough, the estimate is proved. \\

\noindent ${\bf (h_3)}$ There are positive constants $c_1,c_2$ and $ \delta'$  satisfying 
\begin{equation} \label{7}
H(t) \leq c_1t^{\frac{p}{p-1}}, \quad \forall t \geq 0
\end{equation}
and
\begin{equation} \label{8}
H(t)\geq \left\{ 
\begin{array}{c}
c_2 t^{q/(q-1)},\ \mbox{for}\ t \in [0,  \delta'), \\
c_2 t^{p/(p-1)},\ \mbox{for}\ t\geq \delta'.
\end{array}
\right.
\end{equation}

\noindent The proof of ${\bf (h_3)}$ follows with the same type of arguments explored to prove ${\bf (h_2)}$. \\

\noindent \noindent ${\bf (h_4)}$ \, The function $H(t)-\frac{1}{2}h(t)t$ is increasing for $t>0$. \\

\noindent This property is an immediate consequence of the fact that $h \in C^{1}(\mathbb{R})$ and $\frac{h(t)}{t}$ is decreasing for $t>0$. \\

Using  ${\bf (h_0)}$-${\bf (h_3)}$, it is easy to check that $\Psi$ is $C^1(L^{\frac{p}{p-1}}(\Omega),\mathbb{R})$ with 
$$
\Psi'(w)\eta=\int_{\Omega}h(w)\eta dx-\int_{\Omega}\eta Tw dx; \quad \forall w,\eta \in L^{\frac{p}{p-1}}(\Omega).
$$
Moreover, if $w \in L^{\frac{p}{p-1}}(\Omega)$ is a critical point of $\Psi$, then it generates a solution for  $(P)$. Indeed, because for any $ \eta \in L^{\frac{p}{p-1}}(\Omega)$, we know that $\Psi'(w)\eta=0$, or equivalently,
$$
\int_{\Omega}(h(w)-Tw)\eta dx=0,\ \forall \eta \in L^{\frac{p}{p-1}}(\Omega),
$$
implying that
$$
Tw=h(w).
$$
Hence, setting $u=Tw$, we derive that 
$$
\Delta^2 u=\Delta^2 Tw=w=f(h(w))=f(Tw)=f(u).
$$
Furthermore, $u$ also verifies the boundary condition $Bu=0$. Thus, $u$ is a nontrivial solution of $(P)$. Here, it is very important to observe that $u$ is a {\it nodal solution} if, and only if, $w$ is a {\it nodal critical point}, that is, $w^{\pm} \not=0$.

Next, we show that $\Psi$ satisfies the mountain pass geometry.  
\begin{lemma}\label{l1}\mbox{}\\
	\noindent $i)\mbox{There exist}\ \rho, \beta>0\ \mbox{such that}\ \Psi(w)\geq \beta,\ \mbox{for}\ \|w\|= \rho.$\\
	\noindent $ii) \mbox{There exists}\ e\in L^{\frac{p}{p-1}}(\Omega),\ \mbox{with}\ \|e\|> \rho,\ \mbox{such that}\ \Psi(e)<0. $
	
\end{lemma}

\begin{proof}
For each $w \in L^{\frac{p}{p-1}}(\Omega)$, 
$$
\int_{\Omega}H(w)dx\ge c_1\int_{[\left|w(x) \right|\leq \delta']}\left|w(x) \right|^{q/(q-1)}dx+c_2\int_{[\left|w(x) \right|> \delta']}\left|w(x) \right|^{p/(p-1)}dx.
$$
By H\"older's inequality, 
$$
\int_{[\left|w(x) \right|\leq \delta']}\left|w(x) \right|^{p/(p-1)}dx \leq c \left( \int_{[\left|w(x) \right|\leq \delta']}\left|w(x) \right|^{q/(q-1)}dx\right) ^{p(q-1)/(p-1)q}
$$
and since $q \in (2,p]$, if $\|w\|$ is small enough, we see that 
$$
\int_{[\left|w(x) \right|> \delta']}\left|w(x) \right|^{p/(p-1)}dx \geq \left( \int_{[\left|w(x) \right|> \delta']}\left|w(x) \right|^{p/(p-1)}dx\right)^{q(p-1)/p(q-1)}. 
$$
Gathering the last two inequality, we get 
		\begin{align}\Psi(w)\geq &\tilde{c}_1\left( \int_{[\left|w(x) \right|\leq \delta']}\left|w(x) \right|^{p/(p-1)}dx\right)^{q(p-1)/p(q-1)} +&\nonumber\\
			&+c_2\left( \int_{[\left|w(x) \right|> \delta']}\left|w(x) \right|^{p/(p-1)}dx\right)^{q(p-1)/p(q-1)}-c_3\|w\|^2.&\nonumber 
		\end{align}
Recalling that given $\alpha>0$ there is $C>0$ such that  	
$$
A^{\alpha}+B^{\alpha}\geq C(A+B)^{\alpha}, \quad \forall A,B>0,
$$
it follows that  
$$
\Psi(w)\geq  C\|w\|^{q/(q-1)}-c_3\|w\|^{2} \quad \forall w \in L^{\frac{p}{p-1}}(\Omega).  
$$
Once $q>2$, fixing $\rho$ small enough, we find $\beta >0$ such that  
$$
\Psi(w)\geq \beta,\ \mbox{for}\ \|w\|=\rho,
$$	
showing $i)$. To show $ii)$, it is sufficient to see that for each $w \in L^{\frac{p}{p-1}}(\Omega)\setminus \{0\}$ and $t>0$, 
$$
\Psi(tw) \to -\infty \quad \mbox{as} \quad t \to +\infty.
$$		
Here, we have used ${\bf (h_3)}$ and $(T_1)$. 
\end{proof}

\vspace{0.5 cm}

The lemma below will help to prove that $\Psi$ verifies the $(PS)$ condition.

\begin{lemma}\label{l2}
	Let $\{w_n\}$ be a $(PS)_c$ sequence for $\Psi$. Then, $\{w_n\}$ is bounded in $L^{\frac{p}{p-1}}(\Omega)$.
\end{lemma}

\begin{proof}
Since $\{w_n\} \subset L^{\frac{p}{p-1}}(\Omega)$ is a $(PS)_c$ sequence for $\Psi$, we must have  
$$
\Psi(w_n) \rightarrow c\ \mbox{and}\ \Psi'(w_n) \rightarrow 0.
$$
Hence, 
	\begin{equation}\label{9}
		\Psi(w_n)-\frac{1}{2}\Psi'(w_n)w_n\leq
		c+1+\|w_n\|
	\end{equation}
for $n$ large enough. On the other hand, from  (\ref{5}),
	\begin{align}\label{10}
		\Psi(w_n)-\frac{1}{2}\Psi'(w_n)w_n &= \int_{\Omega}\left(
		H(w_n)-\frac{1}{2}h(w_n)w_n\right)dx&\nonumber\\
		&\ge \tilde{C}_{\epsilon}\int_{\Omega}\left|w_n \right|^{\frac{p}{p-1}} dx-\tilde{C}\left| \Omega \right|.
	\end{align}
Gathering $(\ref{9})$ and $(\ref{10})$,
$$
\tilde{C}_{\epsilon}\|w_n\|^{\frac{p}{p-1}}-\tilde{C}\left|\Omega \right| \leq c+1+\|w_n\|,
$$ 
for $n$ large enough. As $p>2$, the last inequality yields $\{w_n\}$ is bounded in
	$L^{\frac{p}{p-1}}(\Omega)$.
\end{proof}

\vspace{0.5 cm}

From the previous lemmas, we are ready to show that $\Psi$ verifies the $(PS)$ condition.

\begin{lemma}\label{l3}
	The functional $\Psi$ satisfies the $(PS)$ condition. 
\end{lemma}

\begin{proof}
	Let $\{w_n\}$ be a $(PS)_c$ sequence for $\Psi$. Then,  
	$$
	\Psi(w_n) \rightarrow c\ \mbox{and}\ \Psi'(w_n)\rightarrow 0.
	$$
Consequently, 
$$
\sup_{\|\eta \|\leq 1}\left|\Psi'(w_n)\eta\right| \rightarrow 0,
$$
or equivalently, 
$$
\sup_{\|\eta \|\leq 1}\left|\int_{\Omega}\left(h(w_n)-Tw_n\right)\eta \, dx\right| \rightarrow 0.
$$
Applying  Riez's Theorem, we can guarantee that 
$$
\left| h(w_n)-Tw_n\right|_{L^{p}(\Omega)}\rightarrow 0.
$$
	
On the other hand, from Lemma \ref{l2}, $\{w_n\}$ is bounded in $L^{\frac{p}{p-1}}(\Omega).$ As  $L^{\frac{p}{p-1}}(\Omega)$ is reflexive, for some subsequence of $\{w_n\}$, still denoted by itself, there is $w \in L^{\frac{p}{p-1}}(\Omega)$ such that 
$$
w_n \rightharpoonup w \quad \mbox{in} \quad L^{\frac{p}{p-1}}(\Omega).
$$
Now, using the compactness of $T$,  we infer that  $Tw_n \rightarrow Tw$ in $L^p(\Omega),$ and so,  
$$
\left| h(w_n)-Tw\right|_{L^{p}(\Omega)}\leq \left| h(w_n)-Tw_n\right|_{L^{p}(\Omega)}+\left|Tw_n-Tw\right|_{L^{p}(\Omega)}\rightarrow 0,
$$
implying that for some subsequence, there is 	 $g \in L^{p}(\Omega)$ such that
\begin{equation}\label{11}
\left|h(w_n)(x)\right|\leq g(x) \quad \mbox{for a.e} \quad x \in \Omega 
\end{equation}
and
\begin{equation}\label{12}
h(w_n(x)) \rightarrow u(x) \quad \mbox{for a.e} \quad x \in \Omega.
\end{equation}
Recalling that $h$ is the inverse of $f$, it follows that	
\begin{equation} \label{13}
w_n(x) \rightarrow f(u(x)):=w(x) \quad \mbox{for a.e} \quad x \in \Omega.
\end{equation}
Combining  (\ref{1}) and (\ref{2}), there exist positives constants $M_1,M_2$ and $\delta"$ such that
	\begin{equation}\label{14}
	\left|h(w_n)\right|\geq \left\{\begin{array}{cc}
	M_1|w_n|^{1/(p-1)}&,\left| w_n\right| > \delta" \\
	M_2|w_n|^{1/(q-1)},&\left| w_n\right|\leq \delta".
	\end{array} \right.
	\end{equation}
Therefore, from  $(\ref{11})-(\ref{14})$, there is $\widetilde{g} \in L^{\frac{p}{p-1}}(\Omega)$ such that  
	$$
	|w_n(x)| \leq \widetilde{g}(x) \quad \mbox{for a.e} \quad x \in \Omega.
	$$
The last inequality combined with Lebesgue's Theorem gives 
$$
w_n \rightarrow w \, \mbox{in}\ L^{\frac{p}{p-1}}(\Omega),
$$ 
finishing the proof. \end{proof}

\begin{theorem} \label{T1}
The functional $\Psi$ has a critical point $w_* \in L^{\frac{p}{p-1}}(\Omega)$, whose the energy is equal to mountain pass level. Moreover, $w_*$ has defined signal, that is, it is positive or negative on $\Omega$. 
\end{theorem}
\begin{proof}
By Lemmas \ref{l1} and \ref{l2}, 	the functional $\Psi$ satisfies the hypotheses of the Mountain Pass Theorem due to Ambrosetti-Rabinowitz \cite{AR}. Thus, the mountain pass level $c$ is a  critical point for $\Psi$, that is, there exists  $w_* \in L^{\frac{p}{p-1}}(\Omega)$ such that $\Psi'(w_*)=0$ and $\Psi(w_*)=c>0.$ Moreover, once $\Psi(0)=0$, we conclude $w_* \neq 0.$ We recall that $c$ is given by 
\begin{equation} \label{15}
	c=\inf_{\gamma \in \Gamma}\max_{t \in [0,1]}\Psi(\gamma(t))>0,
\end{equation}
where
$$
\Gamma=\left\{ \gamma \in C([0,1],L^{\frac{p}{p-1}}(\Omega)) ;\ \gamma(0)=0\ \mbox{and}\ \Psi(\gamma(1))<0\right\}.
$$

Before to continue the proof, we would like to point out that using the same type of arguments found in Willem's book \cite{W}, we can ensure that the mountain pass level $c$ verifies the following equalities 
\begin{equation} \label{16}
c=\inf_{w\in L^{\frac{p}{p-1}}(\Omega)\setminus \{0\}}\sup_{t\geq 0}\Psi(tu)=\inf_{u \in \mathcal{N}}\Psi(u)=\inf_{u \in \mathcal{N}_{\Psi}}\Psi(u)
\end{equation}
where
$$
\mathcal{N}=\{w \in L^{\frac{p}{p-1}}(\Omega)\setminus \{0\}; \Psi'(w)w=0\}
$$
and
$$
\mathcal{N}_{\Psi}=\{w \in L^{\frac{p}{p-1}}(\Omega)\setminus \{0\}; \Psi'(w)=0\}.
$$
The set $\mathcal{N}$ is called the {\it Nehari Manifold} associated with $\Psi$.

Now, we will show that $w_*$ has a defined signal.  Indeed, since
$$
\int_{\Omega}w_*Tw_*dx=\int_{\Omega}(w_*{^+}+w_*{^-})T(w_*{^+}+w_*{^-})dx \leq \int_{\Omega}w_*{^+}Tw_*{^+}dx+\int_{\Omega}w_*{^-}Tw_*{^-}dx,
$$
we have that 
$$
\Psi(w_*)=\max_{t\geq 0}\Psi(tw_*)\geq \Psi(tw_*)\geq \Psi(tw_*{^+})+\Psi(tw_*{^-}), \quad \forall t\geq 0.
$$
Suppose by contradiction that $w_*{^{\pm}}\neq 0,$ then
$$
\int_{\Omega}w_*{^+}Tw_*{^+}dx> 0\ \mbox{and}\ \int_{\Omega}w_*{^-}Tw_*{^-}dx> 0.
$$
Let $t_0^{\pm} \in \mathbb{R}$ be the unique numbers satisfying
$$
\Psi(t_0^{\pm}w_*{^{\pm}})=\max_{t \geq 0}\Psi(tw_*{^{\pm}})>0.
$$
Using the characterization of $c$ mentioned in (\ref{16}), we derive that 
$$
\Psi(t_0^{+}w_*{^{+}}),\Psi(t_0^{-}w_*{^{-}})\geq c.
$$
Hence,
$$
c=\Psi(w_*)\geq \Psi(t_0^{+}w_*{^{+}})+\Psi(t_0^{+}w_*{^{-}})\geq c+\Psi(t_0^{+}w_*{^{-}}),
$$
from it follows that 
$$
\Psi(t_0^{+}w_*{^{-}})\leq 0.
$$
Therefore, 
$$
t_{0}^+>t_{0}^-.
$$
Of a similar way, 
$$
t_{0}^+<t_{0}^-,
$$
obtaining a contradiction. \end{proof}

\vspace{0.5 cm}

\subsection{Ground state solution}
In this section, we will show the existence of ground state solution for $(P)$, that is, a critical point $u \in H$ of $I$ verifying 
$$
I(u)=\inf_{v \in \mathcal{N}_{I}}I(v)
$$
where
$$
\mathcal{N}_{I}=\left\lbrace u \in H,\, I'(u)=0 \right\rbrace.
$$
To this end, the claim below is crucial in our approach

\begin{claim} $w \in L^{\frac{p}{p-1}}(\Omega)$ is a critical point for $\Psi$ if, and only if, $u=Tw$ is a critical point for $I$. Moreover, $\Psi(w)=I(u)$. 
\end{claim}

Indeed, we know that if $w$ is a critical point of $\Psi$, then $u=Tw$ is a critical point of $I$, see page 8 for more details. Now, given a critical point $u \in H$ of $I$ and setting $w_1=f(u)$, we must have 
$$
\Delta^2u=w_1,
$$ 
that is,
$$
Tw_1=u.
$$ 
Consequently,
$$
\int_{\Omega}Tw_1 \eta dx=\int_{\Omega}u \eta dx=\int_{\Omega}h(w_1) \eta dx, \, \forall \eta \in L^{p/(p-1)}(\Omega),
$$
showing that $w_1$ is a critical point of $\Psi$. Furthermore, 
\begin{align*}
I(u)= I(u)-I'(u)u&=\frac{1}{2}\int_{\Omega}|\Delta u|^2dx-\int_{\Omega}F(u) dx-\int_{\Omega}|\Delta u|^2dx+\int_{\Omega}f(u)u dx&\\
&=\int_{\Omega}[f(u)u-F(u)] dx-\frac{1}{2}\int_{\Omega}|\Delta u|^2dx.  &
\end{align*}
Since,
$$
\int_{\Omega}\Delta Tw_1 \Delta \eta dx=\int_{\Omega} w_1 \eta dx, \forall \eta \in H,
$$
fixing $\eta=Tw_1$, we find 
$$
\int_{\Omega}|\Delta u|^2 dx=\int_{\Omega}|\Delta Tw_1|^2 dx=\int_{\Omega} w_1 Tw_1 dx.
$$
By a direct computation, 
$$
H(t)=\int_{0}^{h(t)}rf'(r)dr=h(t)t-\int_{0}^{h(t)}f(r)dr=f(h(t))h(t)-F(h(t)), 
$$
hence,
$$
H(w_1)=f(h(w_1))h(w_1)-F(h(w_1))=f(u)u-F(u),
$$
leading to
$$
I(u)=\Psi(w_1).
$$
Considering
$$
\mathcal{N}_{I}=\left\lbrace u \in H,\, I'(u)=0 \right\rbrace 
$$
and
$$
d=\inf_{u \in  \mathcal{N}_{I}}I(u),
$$
from the previous analysis, we must have $c=d$. Therefore, $u=Tw_*$ is a {\it ground state solution}  for $(P)$, where $w_*$ is the critical point obtained in Theorem \ref{T1}.


\section{Nodal ground state solution}
In this section, we use the dual method to find a nodal ground state solution for $(P)$. To this end, we will look for by a critical point of $\Psi$ in the set 
$$
\mathcal{M}=\left\lbrace w \in L^{p/(p-1)}(\Omega); w^{\pm}\neq 0\ \mbox{and}\ \Psi'(w)w^+=\Psi'(w)w^-=0  \right\rbrace .
$$
More precisely, we intend to prove that there is $w_0 \in \mathcal{M}$ such that
$$
\Psi(w_0)=\inf_{w \in \mathcal{M}}\Psi(w) \quad \mbox{and} \quad \Psi'(w_0)=0.
$$ 
In this case, we have that $u_0=Tw_0$ is a {\it nodal ground state solution} for $(P)$. This conclusion comes from the study made in the Subsection 2.1, because it is easy to prove that
$$
I(u_0)=\min\{I(u)\,:\, u \quad \mbox{is a nodal solution for} \quad (P) \}.
$$

As $\Psi$ has the nonlocal term $\int_{\Omega}wTwdx$, we see that  
$$
\Psi'(w^+)w^+=\int_{\Omega}w^+Tw^-dx <0\ \mbox{and}\ \Psi'(w^-)w^-=\int_{\Omega}w^-Tw^+dx <0.
$$

The above information do not permit to repeat the standard arguments used to get {\it nodal solution} involving the Laplacian operator. Here, we adapt for our case the approach explored in Alves and Souto \cite{Alves1}. 

Next, we will prove some technical lemmas, which are crucial to get the nodal ground state solution. 

\begin{lemma}\label{l4}
	There exists $\rho>0$ such that
	$$\int_{\Omega}wTwdx \geq \rho,\ \forall w \in \mathcal{N}.$$ 
\end{lemma}	
\begin{proof}
Assume by contradiction that there is $(w_n) \subset \mathcal{N}$ such that 
	\begin{equation}\label{17}
	\int_{\Omega}w_nTw_ndx \rightarrow 0.
	\end{equation}
As 
$$
\int_{\Omega}h(w_n)w_ndx=\int_{\Omega}w_nTw_ndx \quad \forall n \in \mathbb{N}
$$
and $h(t)t \geq 0$ for all $t \in \mathbb{R}$, we have that 
$$
h(w_n)w_n \rightarrow 0 \quad \mbox{in} \quad L^1(\Omega).
$$  
Consequently, for some subsequence, still  denoted by itself, 
	\begin{equation} \label{18}
	 h(w_n(x))w_n(x) \rightarrow 0,\ \mbox{a.e. in} \quad \Omega,
	\end{equation}
	and there is $g \in L^1(\Omega)$ such that
	\begin{equation} \label{19}
	 | h(w_n(x))w_n(x)| \leq g(x),\ \mbox{a.e. in} \quad \Omega.
	\end{equation}
Hence, from $(h_1)$ and (\ref{18}),
$$
w_n(x) \rightarrow 0,\ \mbox{a.e. in} \quad \Omega.
$$
Setting 
$$
A_n=\left\lbrace x \in \Omega, |w_n(x)|\geq f(M)\right\rbrace, 
$$ 
by (\ref{1}) and (\ref{19}), there is $K>0$ such that  
$$
\left| w_n(x)\right|\leq \frac{1}{K}g^{\frac{p-1}{p}}(x) \quad  \mbox{a.e. in} \quad A_n.
$$
On the other hand, if $x \notin A_n$, 
$$
\left| w_n(x)\right|\leq f(M).
$$ 
Thereby,
$$
\left| w_n(x)\right|\leq \frac{1}{K}g^{\frac{p-1}{p}}(x)+f(M)\in L^{p/(p-1)}(\Omega),\quad \mbox{a.e. in} \quad \Omega.
$$
The last inequality combined with (\ref{7}) gives   
$$
H(w_n)\leq c\left| w_n(x)\right|^{\frac{p}{p-1}}\leq \left( \frac{1}{k}g^{\frac{p-1}{p}}(x)+f(M)\right)^{\frac{p}{p-1}}\in L^1(\Omega).
$$
As
$$
H(w_n)(x)\rightarrow 0,
$$ 
the Lebesgue's Theorem ensures that  
\begin{equation}\label{20}
\int_{\Omega}H(w_n)dx \rightarrow 0.
\end{equation}
From $(\ref{15})$ and $(\ref{20})$, 
$$
0 < c \leq \Psi(w_n)\rightarrow 0,
$$
which is an absurd.
\end{proof}
\begin{lemma}\label{l5}
There exists $\rho>0$ such that
$$
	\int_{\Omega}w^{\pm}Tw^{\pm}dx \geq \rho, 
$$ 
for all $w\in \mathcal{M}$ with $w^{\pm}\neq 0.$ 
\end{lemma}	
\begin{proof}
Given $w\in \mathcal{M}$, there are unique  $t_{w^+},t_{w^-}\in (0,1)$ such that 
$$
t_{w^+}w^+,t_{w^-}w^-\in \mathcal{N}.
$$ 
Then, by Lemma \ref{l4}, 
$$
\int_{\Omega}t_{w^+}w^{+}T(t_{w^+}w^+)dx \geq \rho.
$$
Once $t_{w^+}<1,$ we derive that
$$
\int_{\Omega}w^{+}Tw^{+}dx \geq \rho. 
$$
Similarly,
$$
\int_{\Omega}w^{-}Tw^{-}dx \geq \rho, 
$$
finishing the proof. 
\end{proof} 

\begin{lemma}\label{l6}
	Let $v\in L^{\frac{p}{p-1}}(\Omega)$ with $v^{\pm}\neq 0.$ Then, there exist $s,t >0$ such that $\Psi'(tv^++sv^-)v^+=0$ and $\Psi'(tv^++sv^-)v^-=0.$
\end{lemma}	

\begin{proof}
	Hereafter, we consider the vetorial field
	$$V(s,t)=\left(\Psi'(tv^++sv^-)tv^+,\Psi'(tv^++sv^-)sv^- \right). $$
	Note that
	\begin{eqnarray*}
		\Psi'(tv^++sv^-)tv^+&=&\int_{\Omega}tv^+h(tv^++sv^-)dx-\int_{\Omega}tv^+T(tv^++sv^-)dx\\
		&=&\int_{\Omega}tv^+h(tv^+)dx-\int_{\Omega}tv^+T(tv^++sv^-)dx.
	\end{eqnarray*}
	Since $v^+\neq 0,$ there is $\alpha>0$ such that $[v^+\geq \alpha]=\{x \in \Omega\,:\, v^+(x) \geq \alpha\}$ 
	has a positive measure. Thereby,
	\begin{equation*}
		\Psi'(tv^++sv^-)tv^+\geq\int_{[v^+\geq \alpha]}tv^+h(tv^+) dx-\int_{\Omega}tv^+T(tv^++sv^-)dx.
	\end{equation*}
	As $h$ is increasing, for $t$ small enough
		\begin{equation*}
		\Psi'(tv^++sv^-)tv^+\geq \int_{[v^+\geq \alpha]}t\alpha(t\alpha)^{1/(q-1)} dx-\int_{\Omega}tv^+T(tv^++sv^-)dx.
	\end{equation*}
Now,  using the linearity of $T$ together with the fact that $\int_{\Omega}v^+T(v^-)dx<0$, we find    
	\begin{equation*}
		\Psi'(tv^++sv^-)tv^+\geq t^{q/(q-1)}\alpha^{q/(q-1)} \left| [v^+\geq \alpha]\right| -\left\| T\right\| t^2\left\| v^+\right\|^2_{L^{p/(p-1)}(\Omega)} .
	\end{equation*}
	Hence, there is $r>0$ small enough such that
	\begin{equation}\label{21}
		\Psi'(rv^++sv^-)rv^+>0, \quad \forall s>0.
	\end{equation}
The same argument works to prove that   
	\begin{equation}\label{22}
		\Psi'(tv^++rv^-)rv^->0, \quad \forall t>0.
	\end{equation}
	On the other hand,
	\begin{equation*}
		\Psi'(tv^++sv^-)tv^+\leq c\int_{\Omega}tv^+\left| tv^+\right|^{1/(p-1)} dx-t^2\int_{\Omega}v^+T(v^+)dx-ts\int_{\Omega}v^+T(v^-)dx.
	\end{equation*}
Once $t,s\geq r,$ it follows that 
	\begin{equation*}
		\Psi'(tv^++sv^-)tv^+\leq c\int_{\Omega}tv^+\left| tv^+\right|^{1/(p-1)} dx-t^2\int_{\Omega}v^+T(v^+)dx-r^2\int_{\Omega}v^+T(v^-)dx,
	\end{equation*}
and so,
$$
\lim_{t\rightarrow +\infty}\Psi'(tv^++sv^-)tv^+=-\infty, \quad \mbox{uniformly in} \quad s \geq r. 
$$ 
Thus, we can to fix  $R>r$ large enough, such that  
	\begin{equation}\label{23}
		\Psi'(Rv^++sv^-)Rv^+<0, \quad \mbox{uniformly in} \quad s \geq r.
	\end{equation} Analogously, 
	\begin{equation}\label{24}
		\Psi'(tv^++Rv^-)Rv^->0 \quad \mbox{uniformly in} \quad t \geq r.
	\end{equation}
	Therefore, from $(\ref{21})-(\ref{24})$, we can apply Miranda Theorem to get $(s,t)\in (r,R)\times (r,R)$ verifying $V(s,t)=0.$
\end{proof}

\vspace{0.5 cm}

In the sequel, for each $v \in L^{\frac{p}{p-1}}(\Omega)$ with $v^{\pm} \neq 0,$ we set \linebreak $h^v:[0,+\infty)\times [0,+\infty) \rightarrow \mathbb{R}$ by
$$
h^v(t,s)=\Psi(tv^++sv^-).
$$ 
\begin{proposition}\label{l7}
	If $w \in \mathcal{M}$, then
	\begin{description}
		\item[i)] $h^w(t,s)<h^w(1,1)=\Psi(w),\ \forall s,t \geq 0\ \mbox{with}\ (s,t)\neq(1,1).$
		\item[ii)]$det(\Phi^w)'(1,1)<0.$
	\end{description}
\end{proposition}
\begin{proof}

First of all, we need to show the following inequality
\begin{claim}\label{l8}
	If $w \in L^{p/(p-1)}$ with $w^{\pm}\neq 0$, then
	$$
	\left(\int_{\Omega}w^+Tw^-dx\right)^2<\left(\int_{\Omega}w^+Tw^+dx\right)\left(\int_{\Omega}w^-Tw^-dx\right).
	$$
\end{claim} 
	Indeed, by positiveness of  $T$, we have the inequality below
	$$
	\int_{\Omega}(tw^++sw^-)T(tw^++sw^-)dx>0,\, (t,s) \neq (0,0),
	$$
	which combined with the symmetry of $T$ gives  
	$$
	t^2\int_{\Omega}w^+Tw^+dx+2st\int_{\Omega}w^+Tw^-dx+s^2\int_{\Omega}w^-Tw^-dx>0.
	$$
	Then, for $s \neq 0,$ 
	$$
	\left( \frac{t}{s}\right)^2 \int_{\Omega}w^+Tw^+dx+2\left( \frac{t}{s}\right)\int_{\Omega}w^+Tw^-dx+\int_{\Omega}w^-Tw^-dx>0.
	$$
	Making $X=\frac{t}{s}$, we deduce that  
	$$ 
	X^2 \int_{\Omega}w^+Tw^+dx+2X\int_{\Omega}w^+Tw^-dx+\int_{\Omega}w^-Tw^-dx>0, \quad \forall X \in \mathbb{R}.
	$$
	From this, the polynomial 
	$$
	P(X)= X^2 \int_{\Omega}w^+Tw^+dx+2X\int_{\Omega}w^+Tw^-dx+\int_{\Omega}w^-Tw^-dx
	$$ 
	does not have real roots. As $\int_{\Omega}w^+Tw^+dx>0$, there is a $X_0 \in \mathbb{R}, $ such that 
\begin{equation} \label{25}
	P(X)\ge P(X_0)>0, \quad  \forall X \in \mathbb{R}.
\end{equation}
	Since $w\in \mathcal{M},$ we know that $\Psi'(w)w^+=\Psi'(w)w^-=0$. Then, $(1,1)$ is a critical point of $h^w$ and the equalities below hold
$$
\int_{\Omega}h(w^+)w^+dx=\int_{\Omega}w^+Tw^+dx+\int_{\Omega}w^+Tw^-dx
$$
and
$$
\int_{\Omega}h(w^-)w^-dx=\int_{\Omega}w^-Tw^+dx+\int_{\Omega}w^-Tw^-dx.
$$
On the other hand, from $(\ref{1})$ and $(\ref{2})$,  
	\begin{align*}
	h^w(t,s)\leq Ct^{p/(p-1)}&\int_{\Omega}\left| w^+\right|^{p/(p-1)}dx+Cs^{p/(p-1)}\int_{\Omega}\left| w^-\right|^{p/(p-1)}dx&\\
	&-t^2\int_{\Omega}w^+Tw^+dx-2ts\int_{\Omega}w^-Tw^+dx-s^2\int_{\Omega}w^-Tw^-dx ,&
	\end{align*}
	where $C$ is a positive constant. The above estimate together with (\ref{25}) guarantees that  
	$$
	h^w(t,s)\rightarrow -\infty,\ \mbox{when}\ \left|(s,t)\right| \rightarrow +\infty.
	$$
Gathering the continuity of $h^w$ with the last limit, we deduce that $h^w$ assumes a global maximum in some point $(a,b).$
	
	Next, we will show that $a,b>0.$ Indeed, if $b=0$
	$$
	\Psi(aw^+) \geq \Psi(tw^+), \quad \forall t >0, 
	$$
	and 
	$$
	\dfrac{\partial h^w}{\partial t}(a,0)=0.
	$$
Then, 
	$$
	\Psi'(aw^+)w^+=0,
	$$ 
or equivalently,
	\begin{equation} \label{26}
	\frac{1}{a}\int_{\Omega}h(aw^+)w^+dx=\int_{\Omega}w^+T(w^+)dx.
	\end{equation}
Recalling that $\Psi'(w)w^+=0,$ we know that 
\begin{equation} \label{27}
	\int_{\Omega}h(w^+)w^+dx<\int_{\Omega}w^+T(w^+)dx.
	\end{equation}
From  (\ref{26})-(\ref{27}),
	$$
	\int_{\Omega}\left[ \frac{h(aw^+)}{aw^+}-\frac{h(w^+)}{w^+}\right](w^+)^2 dx>0.
	$$
Using the fact that $\frac{h(t)}{t}$ is decreasing for $t>0$, the above inequality gives that $a <1.$ On the other hand, note that 
$$
	h^w(a,0)=\Psi(aw^+)=\Psi(aw^+)-\frac{1}{2}\Psi'(aw^+)aw^+ = \int_{\Omega}\left(H(aw^+)-\frac{1}{2}h(aw^+)aw^+\right)dx.
$$
Since $H(t)-\frac{1}{2}h(t)t$ is increasing for $t>0$, $a \in (0,1)$ and  
$$
\int_{\Omega}\left(H(w^-)-\frac{1}{2}h(w^-)w^-\right)dx>0,
$$ 
we have that 
	\begin{align*}
	h^w(a,0)&<\int_{\Omega}\left(H(w^+)-\frac{1}{2}h(w^+)w^+\right)dx&\\
	&<\int_{\Omega}\left(H(w^+)-\frac{1}{2}h(w^+)w^+\right)dx+\int_{\Omega}\left(H(w^-)-\frac{1}{2}h(w^-)w^-\right)dx&\\
	&<\int_{\Omega}\left(H(w^++w^-)-\frac{1}{2}h(w^++w^-)(w^++w^-)\right)dx&\\
	&< \Psi(w)-\frac{1}{2}\Psi'(w)w=\Psi(w)=h^w(1,1),&
	\end{align*}
obtaining a contradiction, because $(a,0)$ is a global maximum point for $h^w$. The same type of argument 
	shows that $a>0$, showing the claim.

The second claim is that $0<a,b\leq 1.$ In fact, since $(a,b)$ is a critical point of  $h^w,$ we have the equalities  
$$
\Psi'(aw^++bw^-)aw^+=0 \quad \mbox{and} \quad \Psi'(aw^++bw^-)bw^-=0 
$$
which load to 
$$
\quad a^2\int_{\Omega}w^+Tw^+dx+ab\int_{\Omega}w^+Tw^-dx=a\int_{\Omega}h(aw^+)w^+dx
$$
and
$$
b^2\int_{\Omega}w^-Tw^-dx+ab\int_{\Omega}w^+Tw^-dx=b\int_{\Omega}h(bw^+)w^+dx.
$$
Without loss of generality, we will suppose that $a \geq b.$ Then,  
$$
ab\int_{\Omega}w^+Tw^-dx\geq a^2\int_{\Omega}w^+Tw^-dx.
$$
Thereby, 
$$
\int_{\Omega}w^+Tw^+dx+\int_{\Omega}w^+Tw^-dx\leq\int_{\Omega}\frac{h(aw^+)}{aw^+}(w^+)^2dx,
$$
and  
$$
\int_{\Omega}w^+Tw^+dx+\int_{\Omega}w^+Tw^-dx=\int_{\Omega}\frac{h(w^+)}{w^+}(w^+)^2dx.
$$
Gathering the above information, we get the inequality
$$
0\leq \int_{\Omega}\left[\frac{h(aw^+)}{aw^+}-\frac{h(w^+)}{w^+}\right](w^+)^2,
$$
which combined with the fact that 	$\frac{h(t)}{t}$ is decreasing for $t>0$ gives $a \leq 1.$
	
	To conclude the proof of item $i)$, we will show that $h^w$ does not have a global maximum in $\left[0,1\right]\times\left[0,1\right]\setminus \left\lbrace (1,1)\right\rbrace .$ Note that
	\begin{align*}
	h^w(a,b)&=\Psi(aw^++bw^-)-\frac{1}{2}\Psi'(aw^++bw^-)(aw^++bw^-)&\\
	&=\int_{\Omega}\left[H(aw^+)-h(aw^+)aw\right]dx+\int_{\Omega}\left[H(bw^-)-h(bw^-)bw^-\right]dx.&
	\end{align*}  
Therefore, 
	\begin{align*}
	h^w(a,b)&< \int_{\Omega}\left[H(w^+)-h(w^+)w\right]dx+\int_{\Omega}\left[H(w^-)-h(w^-)bw^-\right]dx&\\
	& <\int_{\Omega}\left[H(w^++w^-)-h(w^++w^-)(w^++w^-)\right]dx=h^w(1,1),&
	\end{align*}
	proving $i)$. 
	
	To show the item $ii)$, note that 
	$$
	det(\Phi^w)'(1,1)=G(w^+)G(w^-)-\left( \int_{\Omega} w^+Tw^-dx\right)^2,
	$$
	where 
	$$
	G(v)=\int_{\Omega}h'(v)v^2dx-\int_{\Omega}vTvdx.
	$$
	Once, $w \in \mathcal{M}$ and $\int_{\Omega}w^-Tw^+dx=\int_{\Omega}w^+Tw^-dx$, we derive that  
	$$
	\int_{\Omega}w^+Tw^+dx=\int_{\Omega}h(w^+)w^+dx-\int_{\Omega}w^+Tw^-dx
	$$
	and
	$$
	\int_{\Omega}w^-Tw^-dx=\int_{\Omega}h(w^-)w^-dx-\int_{\Omega}w^+Tw^-dx.
	$$
	Hence,
	$$
	G(w^+)=\int_{\Omega}\left[ h'(w^+)(w^+)^2-h(w^+)w^+\right] dx+\int_{\Omega}w^+Tw^-dx.
	$$
	By ${\bf (h_4)}$,  
	$$
	h(t)t>h'(t)t^2,\quad \forall t \neq 0
	$$
	and so, 
	$$
	G(w^+)<\int_{\Omega}w^+Tw^-dx.
	$$
	Similarly,
	$$
	G(w^-)<\int_{\Omega}w^+Tw^-dx.
	$$
From this, 
	$$
	det(\Phi^w)'(1,1)<0
	$$
	\hspace{13.2cm}
\end{proof}

\begin{corollary}\label{c1}
	Let $v\in L^{\frac{p}{p-1}}(\Omega)$ be a function verifying
	$$
	v^{\pm}\neq 0 \quad  \mbox{and} \quad  \Psi'(v)(v^{\pm})\le 0.
	$$
	Then, there are $t,s \in [0,1]$ such that
	$$
	tv^++sv^- \in \mathcal{M}.
	$$
\end{corollary}

\begin{proof}
	An immediate consequence of the arguments used in the proof of Lemma \ref{l7}.
\end{proof}
	\section{Proof of the Theorem \ref{Main Theorem} }
	Hereafter, we denote by $c_{\mathcal{M}}$ the infimum of $\Psi$ in $\mathcal{M}$, that is,
	$$
	c_{\mathcal{M}}=\inf_{w \in \mathcal{M}}\Psi(w).
	$$
	As $\mathcal{M}\subset\mathcal{N},$  we must have  
	$$
	c_{\mathcal{M}}\geq c>0.
	$$ 
	Let  $\{w_n\} \subset \mathcal{M}$ be  such that 
	$$
	\Psi(w_n)\rightarrow c_{\mathcal{M}}.
	$$
	Using well known arguments, we can assume that $\{w_n\}$ is a bounded sequence of $L^{\frac{p}{p-1}}(\Omega)$. Hence, passing to a subsequence if necessary,
$$
w_n \rightharpoonup w\ \mbox{in}\ L^{\frac{p}{p-1}}(\Omega),
$$
for some $w \in L^{\frac{p}{p-1}}(\Omega)$.

\begin{claim}
			$w_n^{\pm} \rightharpoonup w^{\pm}\ \mbox{in}\ L^{\frac{p}{p-1}}(\Omega).$
\end{claim}	
		Indeed, as $\{w_n\}$ is bounded in $L^{\frac{p}{p-1}}(\Omega)$,  $\{w_n^+\}$ and $\{w_n^-\}$ are also bounded. By reflexivity of $L^{\frac{p}{p-1}}(\Omega)$, there exist $w_1,w_2 \in L^{\frac{p}{p-1}}(\Omega)$ such that
		$$
		w_n^+\rightharpoonup w_1 \quad \mbox{and} \quad  w_n^- \rightharpoonup w_2 \quad \mbox{in} \quad L^{\frac{p}{p-1}}(\Omega) 
		$$
		with 
		$$
		w_1(x) \geq 0, \quad w_2(x) \leq 0 \quad \mbox{a.e. in} \quad \Omega \quad \mbox{and} \quad w=w_1+w_2. 
		$$
	By Lemma \ref{l5}, 
		$$
		\int_{\Omega}w_1Tw_1dx,\int_{\Omega}w_2Tw_2dx\geq \rho
		$$
	implying that 
		$$
		w_1, w_2 \neq 0.
		$$ 
		Combining the compactness of $T$ with the maximum principles found in \cite{G}, we have 
		$Tw_n^+ \rightarrow Tw_1$ in $L^{p}(\Omega)$ and $Tw_1>0$ in $\Omega.$ Then, 
		\begin{equation}\label{28}
		Tw_n^+(x) \rightarrow Tw_1(x)\quad \mbox{a.e. in}\ \Omega.
		\end{equation}
		Still from  the compactness of $T$ 
		$$
		w_n^+Tw_n^+ \rightarrow w_1Tw_1\ \mbox{in}\ L^1(\Omega),
		$$ 
	from where it follows that
		\begin{equation}\label{29}
		w_n^+(x)Tw_n^+(x) \rightarrow w_1(x)Tw_1(x) \quad \mbox{a.e. in}\quad  \Omega,
		\end{equation}
		for some subsequence. Now (\ref{28}) and (\ref{29}) combine to give
		$$
		w_n^+(x)=\frac{w_n^+(x)Tw_n^+(x)}{Tw_n^+(x)}\rightarrow \frac{w_1(x)Tw_1(x)}{Tw_1(x)}=w_1(x) \quad \mbox{a.e. in} \quad \Omega. 
		$$
		Similarly,  
		$$
		w_n^-(x)\rightarrow w_2(x) \quad \mbox{a.e. in} \quad \Omega.
		$$
		Notice that, if $w_1(x)>0$, then, $w_n^+(x)>0,$ for $n$ large enough. Hence,
		$$
		w_n(x)=w_n^+(x)\rightarrow w_1(x).
		$$
the same argument works to prove that 
		$$
		w_n(x)=w_n^-(x)\rightarrow w_2(x) \quad \mbox{if} \quad w_2(x)<0.
		$$
	From this, $w_1(x)w_2(x)=0$\ a.e. in  $\Omega$, and  
		$$
		\left\lbrace x \in \Omega; w_2(x)<0\right\rbrace \cap\left\lbrace x \in \Omega; w_1(x)>0\right\rbrace=\emptyset.
		$$
	Therefore,
		$$
		w^+(x)=max\{w(x),0\}=max\{w_1(x)+w_2(x),0\}=w_1(x) \quad \mbox{a.e. in} \quad \Omega
		$$
		and
		$$
		w^-(x)=min\{w(x),0\}=min\{w_1(x)+w_2(x),0\}=w_2(x) \quad \mbox{a.e. in} \quad \Omega,
		$$ 
		finishing the proof of the claim.
		
		From Lemma \ref{l6}, there exist $t,s>0$ such that
		$$
		\Psi'(tw^++sw^-)w^+=\Psi'(tw^++sw^-)w^-=0.
		$$ 
		Now, we will show that $t,s \leq 1$. As  $\Psi'(w_n)w_n^{\pm}=0,$ 
		$$
		\int_{\Omega }w_n^{+}h(w_n^{+})dx=\int_{\Omega }w_n^{+}Tw_n^+dx+\int_{\Omega }w_n^{-}Tw_n^+dx
		$$
		and
		$$
		\int_{\Omega }w_n^{-}h(w_n^{-})dx=\int_{\Omega }w_n^{-}Tw_n^-dx+\int_{\Omega }w_n^{-}Tw_n^+dx.
		$$
		Taking the limit in the above equalities, we obtain
		$$
		\int_{\Omega }w^{+}h(w^{+})dx= \int_{\Omega }w^{+}Tw^+dx+\int_{\Omega }w^{-}Tw^+dx
		$$
		and
		$$
		\int_{\Omega }w^{-}h(w^{-})dx=\int_{\Omega }w^{-}Tw^-dx+\int_{\Omega }w^{-}Tw^+dx.
		$$
		Since $\Psi'(tw^++sw^-)tw^+=0,$ we know that 
		$$
		\int_{\Omega }h(tw^+)tw^+dx=t^2\int_{\Omega }w^{+}Tw^+dx+ts\int_{\Omega }w^{-}Tw^+dx.
		$$ 
		Supposing $t\geq s, $ we find the inequality below  
		$$
		\int_{\Omega }\frac{h(tw^{+})}{tw^{+}}(w^{+})^2dx\geq\int_{\Omega }w^{+}Tw^+dx+\int_{\Omega }w^{-}Tw^+dx\geq \int_{\Omega }\frac{h(w^{+})}{w^+}(w^{+})^2dx,
		$$
		that is,
		$$
		\int_{\Omega }\left[\frac{h(tw^{+})}{tw^{+}}(w^{+})^2 - \frac{h(w^{+})}{w^+}(w^{+})^2\right] dx\geq 0.
		$$
		Once $\frac{h(t)}{t}$ is decreasing for $t>0$, the last inequality ensures that $t\leq 1$, and so, $s\leq 1$. 
		
		Our next step is to show that $\Psi(tw^++sw^-)=c_{\mathcal{M}}.$ To this end, as $tw^++sw^-\in \mathcal{M}$, we must have 
		\begin{align*}
		c_{\mathcal{M}}&\leq \Psi(tw^++sw^-)=\Psi(tw^++sw^-)-\frac{1}{2}\Psi'(tw^++sw^-)(tw^++sw^-)&\\
		&\leq\int_{\Omega}\left[H(tw^++sw^-)-\frac{1}{2}h(tw^++sw^-)(tw^++sw^-)\right]dx,&\\
		&\leq \int_{\Omega}\left[H(tw^+)-\frac{1}{2}h(tw^+)(tw^+)\right]dx+\int_{\Omega}\left[H(sw^-)-\frac{1}{2}h(sw^-)(sw^-)\right]dx.
		\end{align*} 
		Using again that $H(t)-\frac{1}{2}h(t)t$ increasing for $t>0$, we obtain  
		\begin{align*}
		c_{\mathcal{M}}&\leq\int_{\Omega}\left[H(w^+)-\frac{1}{2}h(w^+)(w^+)\right]dx+\int_{\Omega}\left[H(w^-)-\frac{1}{2}h(w^-)(w^-)\right]dx,&
		\end{align*}
		By Fatou's Lemma
		\begin{align*}
		c_{\mathcal{M}}&\leq \liminf_{n\rightarrow +\infty}\int_{\Omega}\left[H(w_n^+)-\frac{1}{2}h(w_n^+)(w_n^+)\right]dx+\liminf_{n\rightarrow +\infty}\int_{\Omega}\left[H(w_n^-)-\frac{1}{2}h(w_n^-)(w_n^-)\right]dx&\\
		&\leq \liminf_{n\rightarrow +\infty}\int_{\Omega}\left[H(w_n^++w_n^-)-\frac{1}{2}h(w_n^++w_n^-)(w_n^++w_n^-)\right]dx,& \\
		&\leq \liminf_{n\rightarrow +\infty}\left[ \Psi(w_n)-\frac{1}{2}\Psi'(w_n)w_n\right]= \liminf_{n\rightarrow +\infty} \Psi(w_n)=c_{\mathcal{M}},& 
		\end{align*}
showing that 
		$$
		c_{\mathcal{M}}=\Psi(tw^++sw^-).
		$$
Setting 	$w_0=tw^++sw^- \in \mathcal{M}$, it follows that
$$
w_0\in \mathcal{M} \quad \mbox{and} \quad \Psi(w_0)=c_\mathcal{M}.
$$

Now, using Proposition \ref{l7} and the same arguments found in \cite[Section 2]{Alves1}, we can infer that $w_0$ is a critical point of  $\Psi$. Thus,  $u=Tw_0$ is a nodal ground state solution for $(P)$.

\vspace{0.2cm}

\section{Final comments} In the present paper we have opted to assume the condition $(f_2)$ to avoid more technicalities, because our main intention were to show in details the idea of the method.


\begin{thebibliography}{99}
	
	\bibitem{Alves2} C.O. Alves, {\it Existence of periodic solutions for a class of systems involving nonlinear wave equations}, Commun. Pure Appl.  Anal. 04 (2005), 487-498.	

	\bibitem {ACM} C.O. Alves, P.C. Carri\~ao and O.H. Miyagaki, 
	{\it On the existence of positive solutions of a perturbed Hamiltonian system in $\mathbb{R}^N$.} 
	J. Math. Anal. Appl. {276} (2002), 673-690. 
	
	\bibitem{Alves1} C. O. Alves and M. A S. Souto, {\it Existence of least energy nodal solution for a 
		Schr\"{o}dinger-Poisson system in bounded domains}, Z. Angew. Math. Phys. 65 (2014), 1153-1166.
		

		
		\bibitem{AR} A. Ambrosetti and P.H. Rabinowitz, {\it Dual variational methods in critical point theory and applications,} J. Funct. Anal. 14 (1973) 349-381
	
	\bibitem{AS} A. Ambrosetti and M. Struwe,  \textit{A note on the problem $\Delta u=\lambda u+|u|^{2^*-2}u,$} Manuscripta Math. {54} (1986),  373-379. 
	
	\bibitem{Weth} T. Bartsch, T. Weth and  M. Willem, {\it Partial symmetry of least energy
		nodal solution to some variational problems}, J. Anal. Math. {96}
	(2005)1-18.
	
	\bibitem{Bartsch} T. Bartsch and T. Weth, {\it Three nodal solutions of singularly perturbed
		elliptic equations on domains without topology}, Ann. Inst. H.
	Poincar\'e Anal. Non Lin\'eaire {22} (2005),  259-281.
	
	\bibitem{Bartsch1} T. Bartsch, Z. Liu and T. Weth, {\it Sign changing solutions of superlinear
		Schr\"{o}dinger equations}, Comm. Partial Differential Equations {29}
	(2004), 25-42.
	
	
	\bibitem{BG} E. Berchio and F. Gazzola, \textit{Positive solutions to a linearly perturbed critical growth biharmonic problem.} Discrete Contin. Dyn. Syst. Ser. S { 4} (2011),  809-823.
	
	\bibitem{BDF} D. Bucur and F. Gazzola, \textit{The first biharmonic Steklov eigenvalue: positivity preserving and shape optimization.} Milan J. Math. {79} (2011), 247-258.
	
	\bibitem{Castro} A. Castro, J. Cossio and J. Neuberger, {\it A sign-changing solution for a
		superlinear Dirichlet problem}, Rocky Mountain J. Math. {27} (1997), 1041-1053.
	
	\bibitem{FG} A. Ferrero and F. Gazzola, \textit{A partially hinged rectangular plate as a model for suspension bridges.} Discrete Contin. Dyn. Syst.  35 (2015), 5879-5908 .
	
	\bibitem{G}F. Gazzola, H. Grunau and G. Sweers, \textit{Polyharmonic boundary value problems.} Lectures notes in mathematics,1991. Springer-Verlag , Berlin, 2010.
	
	\bibitem{GR} H. Grunau and F. Robert,  \textit{Positivity and almost positivity of biharmonic Green's functions under Dirichlet boundary conditions.} Arch. Ration. Mech. Anal. 195 (2010), 865-898. 
	
	\bibitem{GK} C. P. Gupta and Y. C. Kwong, \textit{Biharmonic eigen-value problems and $L^ p$ estimates.} Int. J. Math. Sci. {13} (1990), 469-480.
	
	\bibitem{JQ} T. Jung and Q-Heung Choi, \textit{Nonlinear biharmonic boundary value problem.}  Bound. Value Probl. (2014), 2014:30.
	
	\bibitem{LM} A. C Lazer and P. J. McKenna, \textit{Large-amplitude periodic oscillations in suspension bridges: some new connections with nonlinear analysis.} Siam Rev. {32} (1990),  537-578.
	
	\bibitem{P3} M. T. O. Pimenta, \textit{Radial sign-changing solutions to biharmonic nonlinear Schr\"odinger equations.}   Bound. Value Probl. (2015), 2015:21.
	
	\bibitem{P1} M.T.O. Pimenta and S.H.M. Soares, \textit{Existence and concentration of solutions for a class of biharmonic
	equations.} J. Math. Anal. Appl. {390} (2012), 274-289. 
	
	\bibitem{P2} M.T.O. Pimenta and S.H.M. Soares, \textit{Singularly perturbed biharmonic problems with
	superlinear nonlinearities.} Adv. Differential Equations {19} (2014), 31-50. 
	
	\bibitem{YT} Y. Ye and Chun-Lei Tang, \textit{Existence and multiplicity of solutions for fourth-order elliptic equations in $\mathbb{R}^N$.}  J. Math. Anal. Appl. {406} (2013), 335-351. 
	
	\bibitem{ZWZ} W. Zhang, X. Tang and J. Zhang, \textit{Infinitely many solutions for fourth-order elliptic equations with sign-changing potential.}  Taiwanese J. Math. {18} (2014), 645-659. 
	
	\bibitem{Struwe} M. Struwe, \textit{Variational Methods: Applications to nonlinear partial differential equations and Hamiltonian systems, Springer, Berlin 1990. }
	
	\bibitem{Weth2}T. Weth, \textit{Nodal solutions to superlinear biharmonic equations via decomposition in
	dual cones.} Topol. Methods Nonlinear Anal. {28} (2006), 33-52. 
	
	\bibitem{W} M. Willem, \textit{Minimax Theorems}, Birkh\"auser Boston, MA, 1996.
	
	\bibitem{W1} M. Willem, \textit{Subharmonic oscillations for a semilinear wave equation}, Nonlinear Anal. 09 (1985), 503-514.
	
\end{thebibliography}
\end{document}